\documentclass[11pt]{amsart}

\usepackage{amssymb,amsmath,latexsym}
\usepackage{amsthm}
\usepackage{amsfonts}
\usepackage{enumerate}
\usepackage{booktabs}
\usepackage{mathrsfs}

\usepackage{hyperref}

\usepackage{graphicx,pgf}
\usepackage{float}

\usepackage{caption}
\usepackage{subcaption}

\usepackage{fancyhdr}

\theoremstyle{plain} 
\newtheorem{thm}{Theorem}[section]
\newtheorem{prop}[thm]{Proposition}
\newtheorem{lemma}[thm]{Lemma}

\newtheorem*{remark}{Remark}

\theoremstyle{definition}

\theoremstyle{remark}

\numberwithin{equation}{section}

\newcommand{\mysectionname}{}
\newcommand{\newsection}[1]{\section{#1}\renewcommand{\mysectionname}{\uppercase{#1}}}

\newcommand{\g}[1]{\mathit{G}_{#1}}
\newcommand{\f}[1]{F_{#1}}

\date{November 22, 2014, Revised: October 9, 2015}

\begin{document}
\title{superconvergence to freely infinitely divisible distributions}
\author{Hari Bercovici}
\address{Hari Bercovici: Department of Mathematics, Rawles Hall, 831 East Third Street, Indiana University, Bloomington, Indiana 47405, USA}
\email{bercovic@indiana.edu}
\author{Jiun-Chau Wang}
\address{Jiun-Chau Wang: Department of Mathematics and Statistics, University of Saskatchewan, Saskatoon,
Saskatchewan S7N 5E6, Canada}
\email{jcwang@math.usask.ca}
\author{Ping Zhong}
\address{Ping Zhong: School of Mathematics and Statistics, Wuhan University, No. 299 Ba Yi Road, Wuhan, Hubei, China, 430072; and Department of Mathematics and Statistics, Fylde College, Lancaster University, Lancaster, LA1 4YF, United Kingdom}
\email{p.zhong2@lancaster.ac.uk}

\begin{abstract}
We prove superconvergence results for all freely infinitely divisible distributions. 
Given a nondegenerate freely infinitely divisible distribution $\nu$, let $\mu_n$ be a sequence of probability measures
and let $k_n$ be a sequence of integers tending to infinity such that $\mu_n^{\boxplus k_n}$ converges weakly to $\nu$. We show that
the density $d\mu_n^{\boxplus k_n}/dx$ converges uniformly, as well as in all $L^{p}$-norms for $p>1$, to the density of $\nu$
except possibly in the neighborhood of one point. Applications include the global superconvergence to freely stable laws and that to free compound Poisson laws over the whole real line.
\end{abstract}



\maketitle
\newsection{Introduction}
Consider a sequence $\{X_i\}_{i=1}^{\infty}$ of independent identically distributed
random variables with zero mean and unit variance. The classical central limit theorem states
that variables
\[
  S_n=\frac{X_1+X_2+\cdots+X_n}{\sqrt{n}}
\]
converge in distribution to the standard normal law. Note that the variables $S_n$ might
always be discrete, even though their limit is absolutely continuous. This means that the convergence of $S_n$
to a normal law must be expressed in terms of distribution functions, rather than densities.

Assume now that, instead of being independent, the variables $\{X_i\}_{i=1}^{\infty}$
are \emph{freely} independent in the sense of Voiculescu \cite{Basic}.
We still assume them identically distributed with zero mean and unit variance. Under the additional 
condition that the variables are bounded, it was shown in \cite{BV1995sup} that 
the distribution of $S_n$ is absolutely continuous for sufficiently large $n$, and these densities
converge uniformly, along with all of their derivatives, to the density of the semicircle law
\[
   \frac{1}{2\pi}\sqrt{4-t^2}
\]
on any interval $[a,b]\subset (-2,2)$. This phenomenon was called \emph{superconvergence}
in that paper. The assumption that $X_i$ be bounded was removed in subsequent work of the second 
author \cite{Wang2010}. Even when the variables $X_i$ are not identically distributed, 
but are uniformly bounded, the support of $S_n$ was shown
by Kargin \cite{Kargin2007} to converge to the interval $[-2,2]$ as $n\rightarrow \infty$. See also \cite{mulsuper} for free multiplicative superconvergence results.

The purpose of this paper is to demonstrate that the phenomenon of superconvergence
is not limited to convergence to the semicircle law. 
Consider a nondegenerate probability measure $\nu$ on $\mathbb{R}$, which is infinitely divisible in the 
free sense (that is, $\boxplus$-infinitely divisible). It is known that
its Cauchy transform
\begin{equation}\label{eq:1.1}
   \g{\nu}(z)=\int_{-\infty}^{+\infty}\frac{1}{z-t}\,d\nu(t)
\end{equation}
defined for $\Im z > 0$ extends continuously to all points $z\in\mathbb{R}$ with
at most one exception $t_{\nu}$. The measure $\nu$ is absolutely continuous on 
$\mathbb{R}\backslash \{t_{\nu} \}$
and its density is locally analytic when strictly positive. 
To formulate our result, assume that for every positive integer $n$, we are given $k_n$ freely independent,
identically distributed random variables $X_{n1},X_{n2},\cdots,X_{n k_n}$ such that $\lim_{n\rightarrow \infty}k_n=\infty$ and
the sums
\[
 S_n=X_{n1}+X_{n2}+\cdots+X_{n k_n}
\]
converge in distribution to the measure $\nu$. (Necessary and sufficient conditions for such a convergence to take place are found in \cite{BPata1999}.) Our main result, Theorem 4.1, implies the following statement.
For convenience, we denote by $D_{\nu}$ the singleton $\{t_{\nu} \}$ if this point
exists. Otherwise, $D_{\nu}=\varnothing $. 
\begin{thm}\label{thm:1.1}
Given any open set $U\supset D_{\nu}$, 
the distribution $\nu_n$ of $S_n$ is absolutely continuous on $\mathbb{R}\backslash U$ for sufficiently large $n$,
and the density of $\nu_n$ converges  to the density of $\nu$ uniformly and in $L^{p}$-norms for $p>1$ on $\mathbb{R}\backslash U$. 
\end{thm}
Note that
$U$ can be taken to be empty if $D_{\nu}=\varnothing$. 

In Proposition 5.1, we provide the necessary and sufficient conditions for the existence of the singularity $t_{\nu}$, as well as a formula to compute it when this point exists. These conditions and the formula are further used to investigate the quality of convergence to freely stable and free compound Poisson densities.

To prove this result, we first approximate $\nu_n$ by a closely related 
$\boxplus$-infinitely divisible measure $\rho_n$ and we use the 
fact that $\g{\rho_n}$ is a  conformal map. Related considerations appear in the work of Chistyakov and G{\"o}tze \cite{CG2013b}.

The remainder of this paper is organized as follows. In Section 2, we review some relevant preliminaries on
free convolution and freely infinitely divisible distributions. Section 3 is devoted to describing 
the subordination function appearing in free convolution powers. Section 4 contains the proof of
our main result, and some examples and applications are given in Section 5.

\section{Free convolution and freely infinitely divisible distributions}
Let $\mathbb{C}^+=\{z\in\mathbb{C}:\Im z>0\}$ be the complex upper half-plane, and let $\nu$ be a probability measure on $\mathbb{R}$. Recall that the Cauchy transform $G_{\nu}(z)$ of $\nu$ is defined by \eqref{eq:1.1} for $z \in \mathbb{C}^{+}$. The measure $\nu$ can be recovered as the weak limit of the measures
\[
  d\nu_y(x)=-\frac{1}{\pi}\Im \g{\nu}(x+iy)\,dx,\quad x\in\mathbb{R}, \quad y>0,
\]
as $y\rightarrow 0$, and the atoms of $\nu$ can be calculated as follows:
\begin{equation}\label{eq:2.1}
\lim_{\mathop{y\to 0}}iyG_\nu(\alpha+iy)=\nu(\{\alpha\}),\quad\alpha\in\mathbb{R}.
\end{equation} 

The reciprocal $\f{\nu}=1/\g{\nu}$ is an analytic self-map of $\mathbb{C}^{+}$ and plays a role in the calculation of free convolution. More precisely, for any $\eta > 0$ there exists a positive constant $M=M(\eta,\nu)$ such that the function $\f{\nu}$ has an analytic right inverse $\f{\nu}^{-1}$ (relative to the composition) defined in the truncated cone \[\Gamma_{\eta,M}=\{x+iy:y>M,\text{and}\, |x|<\eta y \}.\] The \emph{Voiculescu transform} $\varphi_{\nu}$ of $\nu$ is then defined as $\varphi_{\nu}(z)=\f{\nu}^{-1}(z)-z$, and for any probability law $\mu$ on $\mathbb{R}$, we have \[\varphi_{\mu\boxplus \nu}(z)=\varphi_{\mu}(z)+\varphi_{\nu}(z)\] for all $z$ in a region of the form $\Gamma_{\eta, M}$ where all three transforms are defined (see \cite{BV1993} for the proof). In this sense, the Voiculescu transform linearizes the free convolution $\boxplus$. 

The set of all finite Borel measures on $\mathbb{R}$ is equipped with the topology of weak convergence from duality with continuous bounded functions. Denoting by $\mathcal{M}$ the class of all Borel probability measures on $\mathbb{R}$, we can translate weak convergence of measures in $\mathcal{M}$ into convergence properties of the corresponding Voiculescu transforms.
We recall the following result from \cite{BPata1999}.
\begin{prop}\label{prop:2.1}
Let $\mu,\mu_{1},\mu_2,\dots$ be measures in $\mathcal{M}$. Then the sequence $\mu_n$ converges weakly to the law $\mu$ if and only if there exist $\eta, M>0$ such that the function $\varphi_{\mu_n}$ are defined on $\Gamma_{\eta,M}$
   for every $n$, $\lim_{n\rightarrow \infty}\varphi_{\mu_n}(iy)=\varphi_{\mu}(iy)$ for every $y>M$, and
   $\varphi_{\mu_n}(iy)=o(y)$ uniformly in $n$ as $y\rightarrow\infty$. 
\end{prop}

A measure $\nu \in \mathcal{M}$ is said to be \emph{$\boxplus$-infinitely divisible} if for every positive
integer $n$, there exists a measure $\nu_{n}\in \mathcal{M}$ such that
\begin{equation}\nonumber
 \nu=\underbrace{\nu_{n}\boxplus\nu_{n}\boxplus\cdots\boxplus\nu_{n}}_{n
 \,\,\text{times}}.
\end{equation}
We denote by $\mathcal{ID}(\boxplus)$ the set of all
$\boxplus$-infinitely divisible measures in $\mathcal{M}$.
It was shown in \cite{BV1993} that $\nu \in \mathcal{ID}(\boxplus)$
if and only if the function $\varphi_{\nu}$ extends analytically to a map from $\mathbb{C}^+$
into $\mathbb{C}^-\cup\mathbb{R}$, in which case there exist a real constant $\gamma$
and a finite Borel measure $\sigma$ on $\mathbb{R}$ such that $\varphi_{\nu}$
has the following \emph{free L\'{e}vy-Khintchine representation}: \[\varphi_{\nu}(z)=\gamma+\int_{\mathbb{R}}\frac{1+tz}{z-t}\,d\sigma(t).\]
The pair $(\gamma, \sigma)$ is uniquely determined.
Conversely, given such a pair $(\gamma, \sigma)$, there exists a unique probability law 
$\nu=\nu_{\boxplus}^{\gamma,\sigma} \in \mathcal{ID}(\boxplus)$
satisfying the above integral formula. We shall call
the pair $(\gamma,\sigma)$ the \emph{free generating pair} for $\nu_{\boxplus}^{\gamma,\sigma}$. Weak convergence of $\boxplus$-infinitely divisible laws can be characterized in terms of their free generating pairs; namely, $\nu_{\boxplus}^{\gamma_n,\sigma_n} \rightarrow \nu_{\boxplus}^{\gamma,\sigma}$ weakly if and only if 
$\gamma_n\rightarrow \gamma$ and $\sigma_n\rightarrow\sigma$ weakly (\cite[Theorem 5.13]{BNT2006}).

We review some useful results related to the $F$-transforms of freely infinitely divisible distributions,
which were proved in \cite{BB2005, Huang}, and are closely related to Biane's work \cite{Biane1997}. 
Given $\nu=\nu^{\gamma,\sigma}_{\boxplus}$ in $\mathcal{ID}(\boxplus)$, the function $\f{\nu}$ is a conformal map, and its inverse is the function \[H_{\nu}(z)=z+\varphi_{\nu}(z)=z+\gamma+\int_{\mathbb{R}}\frac{1+tz}{z-t}\,d\sigma(t), \quad z \in \mathbb{C}^{+}.\] This means that $H_{\nu}(F_{\nu}(z))=z$ for all $z\in\mathbb{C}^+$.
Note that $H_{\nu}:\mathbb{C}^+\rightarrow\mathbb{C}$ is an analytic function satisfying
$\Im H_{\nu}(z)\leq \Im z$ for all $z\in\mathbb{C}^+$.
The following result is a consequence of \cite[Theorem 4.6]{BB2005}.
\begin{prop}\label{prop:2.2}
The function $F_{\nu}$ has a one-to-one continuous extension to $\mathbb{C}^+\cup \mathbb{R}$, and 
it satisfies 
  \begin{equation}\label{eq:2.2}
    |F_{\nu}(z_1)-F_{\nu}(z_2)|\geq \frac{1}{2}|z_1-z_2|,\,\, z_1,z_2 \in \mathbb{C}^+\cup \mathbb{R}.
  \end{equation}
If $\alpha \in \mathbb{R}$ is a point such that $\Im F_{\nu}(\alpha)>0$, then 
$F_{\nu}$ can be continued
analytically to a neighborhood of $\alpha$.
\end{prop}
The inequality (\ref{eq:2.2}) implies that \[|H_{\nu}(z_1)-H_{\nu}(z_2)|\leq 2 |z_1-z_2|,\,\, 
   z_1,z_2\in \Omega_{\nu},\] where $\Omega_{\nu}=\f{\nu}(\mathbb{C}^+)$. The function $H_{\nu}$ has 
a one-to-one continuous extension to the closure $\overline{\Omega_{\nu}}$. This extension is still denoted $H_{\nu}$. Thus, we have the following inversion relationships: \[H_{\nu}\left(F_{\nu}(z)\right)=z, \quad z \in \mathbb{C}^{+}\cup \mathbb{R},\] and \[F_{\nu}\left(H_{\nu}(z)\right)=z, \quad z \in \overline{\Omega_{\nu}}.\]

We describe now the boundary set $\partial\Omega_{\nu}$. 
Given $x \in \mathbb{R}$ and $y>0$,
observe that 
\begin{equation}\nonumber
  \Im H_{\nu}(x+iy)=y\left(1-\int_{\mathbb{R}}\frac{1+t^2}{(t-x)^2+y^2}\,d\sigma (t) \right).
\end{equation} It follows that 
\[
 \Im H_{\nu}(x+iy)=0
\]
if and only if
\begin{equation}\label{eq:2.3}
\int_{\mathbb{R}}\frac{1+t^2}{(t-x)^2+y^2}\,d\sigma (t) =1.
\end{equation}
On the other hand, note that for any $x\in\mathbb{R}$,
the positive function
\begin{equation}\nonumber
  y\mapsto \int_{\mathbb{R}}\frac{1+t^2}{(t-x)^2+y^2}\,d\sigma (t) 
\end{equation} 
is continuous and strictly decreasing in $y$, provided that $\sigma \neq 0$; the case $\sigma =0$ corresponds to a measure $\nu$ which is a point mass. 
Thus, for any $x\in\mathbb{R}$, there exists at most one value $y>0$ satisfying (\ref{eq:2.3}).
It is natural to introduce 
two sets\[A_{\nu}=\{x\in\mathbb{R}:g(x)>1 \}\] and \[B_{\nu}=\mathbb{R} \backslash A_{\nu}= \{x\in\mathbb{R}:g(x)\leq1 \},\]
where the function 
\[
g(x)=\int_{\mathbb{R}}\frac{1+t^2}{(t-x)^2}\,d\sigma (t)=\sup_{y>0}\int_{\mathbb{R}}\frac{1+t^2}{(t-x)^2+y^2}\,d\sigma (t), \quad x \in \mathbb{R},
\] 
is a lower semicontinuous function of $x$, so that $A_{\nu}$ is an open set. 
For $x\in A_{\nu}$, define $u_{\nu}(x)$ to be the unique $y$ in $(0,\infty)$
satisfying (\ref{eq:2.3});
for $x\in B_{\nu}$, set $u_{\nu}(x)=0$.
\begin{prop}\label{prop:2.3}\cite{Huang}
The function $\f{\nu}$ maps $\mathbb{R}$ bicontinuously to the graph 
$\gamma_{\nu}$ of the function $u_{\nu}$, that is,
  \begin{equation}\nonumber
    F_{\nu}(\mathbb{R})=\gamma_{\nu}=\{x+iu_{\nu}(x):x\in\mathbb{R} \}.
  \end{equation}  
In particular, the function $u_{\nu}$
is continuous on $\mathbb{R}$.
\end{prop}
We note for further reference that the set $A_{\nu}$ is merely the collection of all $x\in \mathbb{R}$ such that $u_{\nu}(x)>0$. Moreover, for any $t \in \mathbb{R}$, we have $\Im F_{\nu}(t)>0$ if and only if $\Re F_{\nu}(t)\in A_{\nu}$. The graph $\gamma_{\nu}$ is precisely the boundary set $\partial \Omega_{\nu}$, and one has $\Omega_{\nu}=\{z\in \mathbb{C}^{+}:H_{\nu}(z) \in \mathbb{C}^{+}\}$.
The following result now follows easily from these facts; see also \cite{Biane1997,Huang}.
\begin{prop}\label{prop:2.4}
The function $t\mapsto \Re F_{\nu}(t)$ is a strictly increasing homeomorphism from $\mathbb{R}$ to $\mathbb{R}$.
\end{prop}
As shown in \cite{BV1993}, the measure $\nu$ has at most one atom. From \eqref{eq:2.1}, we see that
$\alpha$ is an atom of $\nu$ if and only if $F_{\nu}(\alpha)=0$ (which gives us the uniqueness of the atom
by Proposition \ref{prop:2.2}) and the Julia-Carath\'eodory derivative $F'_{\nu}(\alpha)$
is finite. The value of this derivative is given by
  \begin{equation}\nonumber
    F'_{\nu}(\alpha)=\frac{1}{\nu(\{ \alpha\})}.
  \end{equation}
By the Stieltjes inversion formula, the density of $\nu$ (relative to Lebesgue measure) is given by
\[\frac{d\nu}{dx}(t)=-\frac{1}{\pi}\Im \g{\nu}(t)=\frac{1}{\pi}\frac{\Im F_{\nu}(t)}{|F_{\nu}(t)|^2},\]
at points other than the possible atom $\alpha$.
(This uses the continuous extension of $\f{\nu}$ to $\mathbb{R}$.) 
\begin{lemma}
Consider a measure $\nu \in \mathcal{ID}(\boxplus)$, and denote by $s_{\nu}$ 
the density of the absolutely continuous part of $\nu$.
We have $\lim_{|t|\rightarrow\infty}s_{\nu}(t)=0$.
\end{lemma}
\begin{proof}
Relation \eqref{eq:2.2} implies that
   \begin{equation}\nonumber
       |F_{\nu}(t)-F_{\nu}(i)|\geq \frac{1}{2}|t-i|>\frac{1}{2}|t|, \quad t\in\mathbb{R},
   \end{equation}
so that $|F_{\nu}(t)|>|t|/3$ for $|t| > 6|F_{\nu}(i)|$.
Then the value of density $s_{\nu}$ at such $t$ can be estimated as follows: 
  \begin{equation}\label{eq:2.4}
  s_{\nu}(t)=\frac{1}{\pi}\frac{\Im F_{\nu}(t)}{|F_{\nu}(t)|^2}\leq \frac{1}{\pi}\frac{1}{|F_{\nu}(t)|}<
       \frac{1}{\pi}\frac{3}{|t|}, \quad |t| > 6|F_{\nu}(i)|.
  \end{equation}
The conclusion follows.
\end{proof} The preceding result shows that if $F_{\nu}(t_{\nu})=0$, then we must have $|t_{\nu}| \leq 6|F_{\nu}(i)|$. Moreover, for any $p>1$ and any neighborhood $U$ of the point $t_{\nu}$, the estimate \eqref{eq:2.4} implies that the $p$-th power $|s_{\nu}|^{\,p}$ is continuous and integrable over $\mathbb{R}\backslash U$. If such a zero $t_{\nu}$ does not exist, then the density $s_{\nu}$ will be a continuous function in the $L^{p}$-space for $p>1$.    

The next result follows from the proof of Theorem 4.6 in \cite{BB2005}.
Here we offer a more direct argument.
\begin{lemma}\label{lemma:2.6}
The derivative of $H_{\nu}$ is nonzero at $z=x+iu_{\nu}(x)$, for any $x\in A_{\nu}$.
\end{lemma}
\begin{proof} 
We have
  \begin{equation}\nonumber
     H_{\nu}'(z)=1-\int_{\mathbb{R}}\frac{1+t^2}{(z-t)^2}d\sigma(t),\,\,z\in \mathbb{C}^+.
  \end{equation}
When $x\in A_{\nu}$ and $z=x+iu_{\nu}(x)$, a straightforward calculation and the definition of $u_{\nu}$ lead to
\[
 \begin{split}
   \left| \int_{\mathbb{R}}\frac{1+t^2}{(z-t)^2} d\sigma(t) \right|
    &<\int_{\mathbb{R}}\frac{1+t^2}{|z-t|^2}d\sigma(t)\\
     &=\int_{\mathbb{R}}\frac{1+t^2}{(t-x)^2+u_{\nu}(x)^2}d\sigma(t)=1,
 \end{split}
\] which implies the desired conclusion. 
\end{proof} 

We conclude this section with a useful result.
\begin{lemma}\label{lemma:2.7}
Consider measures $\nu,\nu_n\in \mathcal{ID}(\boxplus)$,
$n\in \mathbb{N}$, such that
$\nu_n\rightarrow \nu $ weakly as $n\rightarrow \infty$, and let 
$I\subset \mathbb{R}$ be a compact interval such that the limiting density $d\nu/dx$ is bounded away from zero on $I$.
Then the density $d\nu_n/dx$ converges uniformly on $I$ to $d\nu/dx$ as $n\rightarrow\infty$.
\end{lemma}
\begin{proof}
Let $(\gamma,\sigma)$,$(\gamma_n,\sigma_n) $ be the free generating pairs of $\nu$ and $\nu_n$, respectively. As seen earlier, $\gamma_n\rightarrow\gamma$ 
and $\sigma_n\rightarrow \sigma$ weakly as $n\rightarrow \infty$. Thus, the sequence $H_{\nu_n}$ converges to the function $H_{\nu}$ uniformly on compact subsets of $\mathbb{C}^+$.

It is clear that $\Re F_{\nu}(I) \subset A_{\nu}$. Thus, 
by Lemma \ref{lemma:2.6},
$H'_{\nu}(z)\neq 0$ for $z\in F_{\nu}(I)$, and 
its inverse function $F_{\nu}$ has a conformal continuation to a neighborhood
of $I$. 
Expressing inverse functions using the Cauchy integral formula, we 
conclude that, for large $n$, $F_{\nu_n}$ also 
has a conformal continuation to a neighborhood of $I$. Moreover,
these continuations converge uniformly on $I$ to the continuation of 
$F_{\nu}$. Since $0\notin F_{\nu}(I)$, the lemma follows from 
the Stieltjes inversion formula.
\end{proof}

\section{free convolution powers and subordination functions}
Given two probability measures $\mu_1$ and $\mu_2$ on $\mathbb{R}$, 
there exist two unique analytic functions $\omega_1,\omega_2:\mathbb{C}^+\rightarrow\mathbb{C}^+$ such that
$F_{\mu_1\boxplus \mu_2}(z)=F_{\mu_1}(\omega_1(z))=F_{\mu_2}(\omega_2(z))$ and 
  \begin{equation}\nonumber
      F_{\mu_1\boxplus\mu_2}(z)
=\omega_1(z)+\omega_2(z)-z  \end{equation}
for all $z\in \mathbb{C}^+$ (see \cite{DVV1993, Biane1998}).

Consider now a sequence $\{\mu_{n}\}_{n=1}^{\infty}$ in $\mathcal{M}$ and positive integers $k_n \geq 2$, and denote by $\mu_{n}^{\boxplus k_n}$ the $k_n$-fold
free convolution power of $\mu_{n}$. 
It is known that $\mu_{n}^{\boxplus k_n}$ has at most one atom and otherwise $\mu_{n}^{\boxplus k_n}$ is absolutely continuous \cite{BB2005}. The analytic subordination for these free convolution powers was also studied in \cite{BB2005}. Thus, let $\omega_n:\mathbb{C}^+\rightarrow\mathbb{C}^+$ be the subordination function of
$F_{\mu_{n}^{\boxplus k_n}}$ with respect to $F_{\mu_{n}}$, that is, $F_{\mu_{n}^{\boxplus k_n}}(z)=F_{\mu_{n}}(\omega_n(z))$.
Then we have 
   \begin{equation}\label{eq:3.1}
       F_{\mu_{n}^{\boxplus k_n}}(z)=F_{\mu_{n}}(\omega_n(z))=\omega_n(z)+\frac{1}{k_n-1}(\omega_n(z)-z), \quad z \in \mathbb{C}^{+}.
   \end{equation}
Equation (\ref{eq:3.1}) implies that the inverse function \[\omega_n^{-1}(z)=z+(k_n-1)(z-F_{\mu_{n}}(z))\] for $z\in \Gamma_{\eta,M}$, where $\eta,M$ are positive constants. 
On the other hand, the function $\omega_n$ can be regarded as the $F$-transform of a unique probability measure on $\mathbb{R}$
by the characterization of $F$-transforms (see \cite[Proposition 5.2]{BV1993}).
Let $\rho_{n}$ be the probability measure
on $\mathbb{R}$ such that $\omega_n(z)=F_{\rho_{n}}(z)$, so that
\begin{equation}\label{eq:3.2}
   \varphi_{\rho_{n}}(z)=(k_n-1)(z-F_{\mu_{n}}(z)).
\end{equation} This implies
that the measure $\rho_{n}$ is $\boxplus$-infinitely divisible. 
In particular, the function $\omega_n$ extends continuously to $\mathbb{C}^+\cup\mathbb{R}$ and so does the function $F_{\mu_{n}^{\boxplus k_n}}$ by \eqref{eq:3.1}.

Denote by $E_{\mu}(z)=z-F_{\mu}(z)$ the self-energy of $\mu$. Given two measures $\mu_1,\mu_2 \in \mathcal{M}$,
their Boolean convolution $\mu_1\uplus \mu_2$, introduced in \cite{RR1997}, is the unique probability measure on $\mathbb{R}$ satisfying \[E_{\mu_1\uplus\mu_2}(z)=E_{\mu_1}(z)+E_{\mu_2}(z), \quad z \in \mathbb{C}^{+}.\]
Every probability measure on $\mathbb{R}$ is $\uplus$-infinitely divisible.
Given a measure $\nu \in \mathcal{M}$,
the function $E_{\nu}$ is a map from $\mathbb{C}^+$ to $\mathbb{C}^-\cup\mathbb{R}$ and satisfies
$E_{\nu}(iy)/iy \rightarrow 0$ as $y\rightarrow \infty$. (The latter limit actually holds uniformly for $\nu$ in any tight family of probability measures \cite{BV1993}.) Thus, $E_{\nu}$ admits a unique \emph{Nevanlinna representation}: \[E_{\nu}(z)= \gamma + \int_{\mathbb{R}}\frac{1+tz}{z-t}\,d\sigma(t), \quad z \in \mathbb{C}^{+}.\] Conversely, every such  formula defines an analytic function which is of the form $E_{\nu}$ for a unique probability measure $\nu$. We will write $\nu=\nu_{\uplus}^{\gamma,\sigma}$ to indicate this correspondence. Note that $E_{\nu_{\uplus}^{\gamma,\sigma}}(z)=\varphi_{\nu_{\boxplus}^{\gamma,\sigma}}(z)$, and that the map $\nu_{\boxplus}^{\gamma,\sigma}\rightarrow \nu_{\uplus}^{\gamma,\sigma}$ is a bijective map from the set $\mathcal{ID}(\boxplus)$ into the set $\mathcal{M}$. Finally, it is easy to verify that if a sequence $\nu_{n}$ converges weakly to a law $\nu$ in $\mathcal{M}$, then the limit $\lim_{n\rightarrow \infty}E_{\nu_{n}}(z)=E_{\nu}(z)$ holds for $z \in \mathbb{C}^{+}$.

We record for further use the following result from \cite[Theorem 6.3]{BPata1999}.
\begin{thm}\label{thm:3.1}
Fix a free generating pair $(\gamma, \sigma)$,
a sequence $\{\mu_n\}_{n=1}^{\infty}$ in $\mathcal{M}$, and 
a sequence $\{k_n\}_{n=1}^{\infty}$ of unbounded positive integers. Then the sequence $\mu_n^{\boxplus k_n}$ converges weakly to $\nu_{\boxplus}^{\gamma,\sigma}$ if and only if the sequence $\mu_n^{\uplus k_n}$ converges weakly to $\nu_{\uplus}^{\gamma,\sigma}$.
\end{thm}

Boolean limit theorems are used in the proof of the following result.
\begin{prop}\label{prop:3.2}
Let $\{\mu_n\}_{n=1}^{\infty} \subset \mathcal{M}$ and let $\{k_n\}_{n=1}^{\infty} \subset \mathbb{N}$ such that $\lim_{n \rightarrow \infty}k_n=\infty$. Suppose the sequence $\mu_n^{\boxplus k_n}$ converges weakly to a law $\nu \in \mathcal{ID}(\boxplus)$.
For each $n$, choose $\rho_n\in\mathcal{ID}(\boxplus) $, such that
    \begin{equation}\nonumber
       \f{\mu_n^{\boxplus k_n}}(z)=\f{\mu_n}(\f{\rho_n}(z)), \quad z \in \mathbb{C}^{+}.
    \end{equation}
Then $\rho_n\rightarrow \nu$ weakly.    
\end{prop}
\begin{proof}
Assume that $(\gamma, \sigma)$ is the free generating pair of $\nu$.
By Proposition \ref{prop:2.1}, the
weak convergence $\mu_n^{\boxplus k_n}\rightarrow\nu_{\boxplus}^{\gamma,\sigma}$ implies 
the existence of $M>0$ such that \[\lim_{n\rightarrow \infty}k_n\varphi_{\mu_n}(iy)=\varphi_{\nu_{\boxplus}^{\gamma,\sigma}}(iy)\]
for all $y>M$, and $k_n\varphi_{\mu_n}(iy)=o(y)$ uniformly in $n$ as $y \rightarrow \infty$. 
In particular, it follows that the sequence $\mu_n$ converges weakly to the unit point mass at $0$. On the other hand, Theorem \ref{thm:3.1} shows that ${\mu_n^{\uplus k_n}}\rightarrow \nu_{\uplus}^{\gamma,\sigma}$ weakly.

By \eqref{eq:3.2}, we have \[\varphi_{\rho_n}(z)=E_{\mu_n^{\uplus k_n}}(z)-E_{\mu_n}(z), \quad z \in \mathbb{C}^{+}.\] Since the two sequences $\{\mu_n^{\uplus k_n}\}_{n=1}^{\infty}$ and $\{\mu_n\}_{n=1}^{\infty}$ are both tight, the last formula implies that $\varphi_{\rho_n}(iy)=o(y)$ uniformly in $n$ as $y\rightarrow \infty$. To determine the limit of $\{\rho_{n}\}_{n=1}^{\infty}$, we calculate \[\lim_{n\rightarrow \infty}\varphi_{\rho_n}(iy)=\lim_{n\rightarrow \infty}[E_{\mu_n^{\uplus k_n}}(iy)-E_{\mu_n}(iy)]=E_{\nu_{\uplus}^{\gamma,\sigma}}(iy)=\varphi_{\nu_{\boxplus}^{\gamma,\sigma}}(iy) \] for every $y>M$. The desired conclusion follows from Proposition \ref{prop:2.1}.
\end{proof}

\section{the main result}
In the following statement, $F_{\nu}$ is viewed as a continuous function 
defined on $\mathbb{C}^+\cup \mathbb{R}$. 
\begin{thm}\label{thm:4.1}
Consider a nondegenerate $\boxplus$-infinitely divisible distribution $\nu$ on $\mathbb{R}$, a sequence $\{\mu_n\}_{n=1}^{\infty}$ of probability measures on $\mathbb{R}$, and a sequence $\{k_n\}_{n=1}^{\infty}$ of positive integers tending to infinity such that the sequence $\mu_n^{\boxplus k_n}$ converges weakly to $\nu$. 
   \begin{enumerate}[$(1)$]
     \item If $0 \notin F_{\nu}(\mathbb{R})$, then the measure $\nu$ has no atom and
      there exists $N>0$ such that the measure $\mu_n^{\boxplus k_n}$ 
     is Lebesgue absolutely continuous with a continuous density on $\mathbb{R}$ for every $n\geq N$. Moreover, the density of the measure $\mu_n^{\boxplus k_n}$ converges uniformly on $\mathbb{R}$ to the density of the measure $\nu$.
     \item If $0\in F_{\nu}(\mathbb{R})$, and $U\subset \mathbb{R}$ is an open interval containing the singleton $F_{\nu}^{-1}(\{0\})$,
      then there exists $N>0$ such that the measure $\mu_n^{\boxplus k_n}$ is absolutely continuous with a continuous density on $\mathbb{R}\setminus U$ for $n \geq N$. Moreover, the density of the measure $\mu_n^{\boxplus k_n}$ converges uniformly on $\mathbb{R}\setminus U$ to the density of the measure $\nu$.
      \item In all cases, the limit \[\lim_{n\rightarrow \infty}\left\lVert\frac{d\mu_n^{\boxplus k_n}}{dx}-\frac{d\nu}{dx}\right\rVert_{L^{p}\left(\mathbb{R}\setminus U\right)}=0\] holds for $p>1$, with $U=\varnothing$ in case \emph{(1)}. 
   \end{enumerate} 
\end{thm}
\begin{remark}
\emph{The condition that $0\in F_{\nu}(\mathbb{R})$ 
is necessary for $\nu$ to have an atom, but it is not sufficient (see Proposition 5.1).
If $F_{\nu}(t_{\nu})=0$, then the function $\g{\nu}$ 
extends continuously to all points $t\in\mathbb{R}\setminus \{t_{\nu}\}$.
Theorem 1.1 follows from Theorem 4.1 and
this observation.}
\end{remark}
\begin{proof}
As seen earlier, there exist measures
$\rho_n\in \mathcal{ID}(\boxplus)$ satisfying
   \begin{equation}\nonumber 
      F_{\mu_n^{\boxplus k_n}}(z)=F_{\mu_n}(F_{\rho_n}(z)), \quad z\in \mathbb{C}^+.
   \end{equation}
To each $n$, denote by $s_n$ and $s$ the density of the absolutely continuous part of $\mu_n^{\boxplus k_n}$ and that of $\nu$, respectively.
Relation (3.1) shows that
$|F_{\mu_n^{\boxplus k_n}}-F_{\rho_n}|$ is small
relative to $|F_{\rho_n}|$. Therefore, it suffices to focus on the asymptotic behavior of $\f{\rho_n}$.

Given $\varepsilon>0$, we first prove that there exists $M>0$ such that $\left|s_n(t)-s(t)\right|< \varepsilon$ for $|t|>M$ and for sufficiently large $n$. 
Since the measures $\rho_n$ converge weakly to $\nu$ by Proposition 3.2, we have $|F_{\rho_n}(i)| \rightarrow |F_{\nu}(i)|$ as $n\rightarrow \infty$. In the sequel, we shall only consider the integers $n$ which satisfy the following two conditions: \[k_n>13 \quad \text{and} \quad 9|F_{\nu}(i)|>6|F_{\rho_{n}}(i)|.\] Applying the estimate (2.4) to $\rho_{n}$, we have $|F_{\rho_n}(t)|> |t|/3$ for all such $n$ and for $|t| > 9|F_{\nu}(i)|$. It follows from (3.1) that
$|F_{\mu_n^{\boxplus k_n}}(t)|>|t|/4$ for the same $n$ and $t$. Combining this with another application of (2.4) to the density $s$, we get    
    \begin{equation}\label{eq:4.1}
       \left|s_n(t)-s(t)\right|< \frac{7}{\pi}\frac{1}{|t|},\quad |t| > 9|F_{\nu}(i)|, 
    \end{equation} for these large $n$. Therefore, the desired cutoff constant $M$ can be chosen as \[M=\max\{ 9|F_{\nu}(i)|,7/\varepsilon\pi\}.\] 
    
    We conclude that it suffices to prove the uniform convergence of $s_n$ to $s$ on a set 
of the form $I\setminus U$, where $I=[-M,M]$. To this purpose, fix then $I=[-M,M]$ with $M>0$, and let $\delta >0$ be arbitrary but fixed. Recall that the map
\[
 t\mapsto \Re F_{\nu}(t)
\]
is an increasing homeomorphism of $\mathbb{R}$.
Thus, the set 
\[
\begin{split}
 J &=\{x\in \mathbb{R}: x\in \Re F_{\nu}(I) \}\\
 &=\{x\in \mathbb{R}:\Re F_{\nu}(-M)\leq x \leq \Re F_{\nu}(M) \}
 \end{split}
\]
is a compact interval. Set 
\[
\Gamma=\{x\in J: u_{\nu}(x)\geq\delta \}
\]
and 
\[
\Delta=\{x\in J:u_{\nu}(x)>\delta/2\}.
\]
We have $\Gamma\subset\Delta\subset J$, $\Gamma$ is closed,
and $\Delta$ is relatively open in $J$. We conclude that $\Gamma$
is contained in the union of finitely many connected components of $\Delta$. Taking the closure of those components, we find a finite family
$J_1,J_2,\cdots,J_K$ of pairwise disjoint, closed intervals such that
\[
\Gamma\subset \bigcup_{1 \leq \ell \leq K} J_{\ell} \subset \overline{\Delta}.
\]
We have $u_{\nu}\geq\delta/2$ on the union $\bigcup_{_{1 \leq \ell \leq K}}J_{\ell}$ and 
$u_{\nu}\leq \delta$ on the complement 
$J'=J\setminus (\bigcup_{_{1\leq \ell \leq K}}J_{\ell})$.

Denote 
$I_{\ell}=\{t\in I: \Re F_{\nu}(t)\in J_{\ell} \}$ for each 
$1\leq \ell \leq K$.
Note that 
\[
\Im F_{\nu}(t)\geq\delta/2
\]
for each $t\in \bigcup_{_{1\leq\ell\leq K}}I_{\ell}$.
Thus, the density of $\nu$ is bounded away from zero
on $\bigcup_{_{1\leq\ell\leq K}}I_{\ell}$.
From Lemma 2.7, we see that 
the functions $F_{\nu}$ 
and $F_{\rho_n}$ both extend analytically to a neighborhood of the set $\bigcup_{_{1\leq\ell\leq K}}I_{\ell}$ for sufficiently large $n$. These extensions are injective. Moreover, the convergence $F_{\rho_n}\rightarrow F_{\nu}$ holds uniformly in that neighborhood. By virtue of (3.1), we conclude that the functions $F_{\mu_n^{\boxplus k_n}}$ will have the same behavior on the set $\bigcup_{_{1\leq\ell\leq K}}I_{\ell}$ as $n \rightarrow \infty$. It follows that the measure $\mu_n^{\boxplus k_n}$ has no atom in the union $\bigcup_{_{1\leq\ell\leq K}}I_{\ell}$ for large $n$ and $s_n \rightarrow s$ uniformly on this set by the Stieltjes inversion formula. 

We prove next the uniform convergence on the set $I'$ (or on $I' \setminus U$), where
\begin{equation}\label{eq:4.2}
I'=\{t\in I: \Re F_{\nu}(t) \in J' \}
=I\setminus \left(\bigcup_{_{1\leq\ell\leq K}}I_{\ell}\right).
\end{equation}
We claim that 
\begin{equation}\label{eq:4.3}
\sup_{x\in J'}u_{\rho_n}(x)\leq 2\delta
\end{equation}
for sufficiently large $n$. Assume, to get a contradiction, that there exist positive integers $n_1<n_2<\cdots\rightarrow \infty$ 
and points $x_{1},x_{2},\cdots\in J'$ such that $u_{\rho_{n_k}}(x_{k})> 2\delta$. 
By the definition of $u_{\rho_{n}}$ given in Section 2, we have
   \begin{equation}\label{eq:4.4}
      \int_{\mathbb{R}}\frac{1+t^2}{(t-x_{k})^2+u_{\rho_{n_k}}(x_{k})^2}\,d\sigma_{n_k}(t)=1, \quad k\geq 1, 
   \end{equation} where $\sigma_{n_k}$ is the free generating measure of $\rho_{n_k}$. 
By passing to a subsequence if necessary, we assume that $x_{k}\rightarrow x_0\in\overline{J'}$ as $k\rightarrow \infty$.
Then, denoting $\nu=\nu_{\boxplus}^{\gamma,\sigma}$, the identity (4.4) and the fact that $\sigma_n \rightarrow \sigma$ weakly imply that \[1\leq \int_{\mathbb{R}}\frac{1+t^2}{(t-x_{k})^2+(2\delta)^2}\,d\sigma_{n_k}(t)
       \rightarrow \int_{\mathbb{R}}\frac{1+t^2}{(t-x_0)^2+(2\delta)^2}\,d\sigma(t)\]
as $k\rightarrow\infty$.
We conclude that $u_{\nu}(x_0)\geq 2\delta$, which is in contradiction to the fact that $x_0\in \overline{J'}$. Thus, the estimate (4.3) is proved. 

The rest of the proof is divided into two cases according to whether 
$U=\varnothing$ or $U\neq \varnothing$. 
By translating the measure $\nu$ if necessary, we may assume that $\Re F_{\nu}(0)=0$. 

Case $(1)$: $0\notin F_{\nu}(\mathbb{R})$ and $U=\varnothing$. In this case, $u_{\nu}(0)>0$ and thus $0\in A_{\nu}$. Since the set $A_{\nu}$ is open, there exists a small number $a>0$ such that the interval $[-4a,4a]$ is contained in $A_{\nu}$. By considering a smaller $\delta$ if necessary, we may assume further that 
    \begin{equation}\label{eq:4.5}
       [-4a,4a ]\subset \bigcup_{1\leq\ell\leq K}J_{\ell}.
    \end{equation}
Since the map $t\mapsto \Re F_{\nu}(t)$ is an increasing homeomorphism of $\mathbb{R}$, the uniform convergence of 
$F_{\rho_n}\rightarrow F_{\nu}$
on $\bigcup_{_{1\leq\ell\leq K}}I_{\ell}$ implies that there exists some integer $N>0$ such that \[\left[-2a,2a\right] \subset
       \left\{\Re F_{\rho_n}(t):t\in \bigcup_{_{1\leq\ell\leq K}}I_{\ell} \right\}, \quad n\geq N.\] Since the map $t\mapsto\Re F_{\rho_n}(t)$ is also
a homeomorphism of the same nature, we have
    \begin{equation}\nonumber
       \inf_{t\in I'}|\Re F_{\rho_n}(t)|\geq 2a, \quad n \geq N, 
    \end{equation}
by recalling the definition (4.2) of the complement $I^{\prime}$. 
Using (3.1) and enlarging $N$ if necessary,
we conclude that
    \begin{equation}\label{eq:4.6}
       \inf_{t\in I'}|\Re F_{\mu_n^{\boxplus k_n}}(t)|\geq a, \quad n \geq N. 
    \end{equation} Further enlarging $N$,
the inequality (4.3) and  the relation (3.1)  imply that
    \begin{equation}\label{eq:4.7}
       \Im F_{\mu_n^{\boxplus k_n}}(t) \leq 3\delta, \quad t\in I', \quad n\geq N.
    \end{equation} From (4.6) and (4.7), we see that
\[
0\leq s_n(t)\leq \frac{3\delta}{a^2 \pi}
\]
for $t\in I'$ and $n \geq N$. On the other hand, the relation (4.5)
and the fact that $u_{\nu}\leq \delta$ on $J'$ yield
\[
0\leq s(t)\leq \frac{\delta}{16a^2 \pi}
\] 
for $t\in I'$. As the parameter $\delta$ can be arbitrarily small, we have proved the uniform convergence
of $s_n\rightarrow s$ on $I'$.
This finishes the proof of Part (1).

Case $(2)$: $0\in F_{\nu}(\mathbb{R})$. In this case, $u_{\nu}(0)=0$ and $F_{\nu}(0)=0=H_{\nu}(0)$ by our normalization.
Let $a_n$ be the unique real value such that $\Re F_{\rho_n}(a_n)=0$ (and hence $F_{\rho_n}(a_n)=iu_{\rho_{n}}(0)$). 
We first show that $a_n$ is small for large $n$. 
Toward this end, we write $U=(-2b,2b)$ where $b>0$ and set $c=b/5$. Observe that
    \begin{equation}\nonumber
       \lim_{n \rightarrow \infty}H_{\rho_n}(ic)=H_{\nu}(ic)\in \mathbb{C}^{+}
    \end{equation}
and
\begin{equation}\nonumber
    |H_{\nu}(ic)|= |H_{\nu}(ic)-H_{\nu}(0)|\leq 2c. 
\end{equation} Since the domain $\Omega_{\rho_{n}}=\{z\in \mathbb{C}^{+}:H_{\rho_{n}}(z) \in \mathbb{C}^{+}\}$, we conclude that exists an integer $N>0$ such that $ic\in \Omega_{\rho_n}$ for all $n \geq N$. Consequently, we have $u_{\rho_n}(0)<c$ for such $n$. Observe that \[|H_{\rho_n}(ic)-a_n|=|H_{\rho_n}(ic)-H_{\rho_n}(iu_{\rho_n}(0))|\leq 2(c-u_{\rho_n}(0))\leq 2c\] for all $n \geq N$. (Notice that we have used the inversion relationship $a_n=H_{\rho_n}\left(F_{\rho_{n}}(a_n)\right)$ here.) Therefore,  by enlarging $N$ if necessary, we conclude that  
$|a_n|<5c=b$ for $n \geq N$.

Now, (2.2) shows that for any $t \in I' \setminus U$ and $n \geq N$, we have \[|F_{\rho_n}(t)-F_{\rho_n}(a_n)|\geq \frac{1}{2}|t-a_n|>\frac{b}{2}.\] This implies further that
\[
 |F_{\rho_n}(t)|>\frac{b}{2}-|F_{\rho_n}(a_n)|=\frac{b}{2}-|u_{\rho_n}(0)|> 
 \frac{b}{4}, \quad t \in I' \setminus U, \quad n \geq N.
\] In other words, for such values of $t$ and $n$, $ |F_{\rho_n}(t)|$ is always bounded away from zero.  
Then an argument similar to the proof of Case (1) yields the absolute continuity of the free convolution $\mu_n^{\boxplus k_n}$ and the uniform convergence $s_n\rightarrow s$ on
$I' \setminus U$, finishing the proof of Part (2).

Finally, the $L^{p}$-convergence result in Part (3) follows from the estimate (4.1) and the dominated convergence theorem.  
\end{proof}
\begin{remark} [Local analyticity and approximation]
   \emph{An important feature of superconvergence is the analyticity properties of the distributions in the limiting process. Indeed, under the weak convergence assumption of Theorem 4.1, if $I$ is a finite interval on which the limit density $d\nu/dx$ is bounded away from zero (and hence it admits an analytic continuation to a neighborhood of $I$), then the restriction of the free convolution $\mu_n^{\boxplus k_n}$ on $I$ becomes absolutely continuous in finite time and its density continues analytically to a neighborhood of $I$. Moreover, these extensions can be approximated uniformly by the analytic continuation of $d\nu/dx$ on $I$, thanks to Lemma 2.7 and the identity (3.1).}
\end{remark}

\section{applications}
In this section, we apply our main result to some of the most important limit theorems in free probability. We begin by examining the geometric condition: $0 \in F_{\nu}\left(\mathbb{R}\right)$. Note that the singular integral in the following result takes values in $(0,\infty]$.
\begin{prop} Let $\nu=\nu_{\boxplus}^{\gamma, \sigma}$ be a nondegenerate law in $\mathcal{ID}(\boxplus)$. We have: 
\begin{enumerate} [$(1)$]
\item $0 \in F_{\nu}\left(\mathbb{R}\right)$ if and only if \begin{equation}\label{eq:5.1} L=\sup_{\varepsilon > 0}\frac{-\Im \varphi_{\nu}(i\varepsilon)}{\varepsilon} =\int_{\mathbb{R}}\frac{1+t^2}{t^2}\,d\sigma(t) \leq 1.\end{equation} In this case, the value of the unique zero $t_{\nu}$ of $F_{\nu}$ is given by \[t_{\nu}=\gamma-\int_{\mathbb{R}}\frac{1}{t}\,d\sigma(t).\] 
\item $\nu\left(\{t_{\nu}\}\right)>0$ if and only if $L<1$, and we have $\nu\left(\{t_{\nu}\}\right)=1-L$ in this case.
\end{enumerate} 
\end{prop}
\begin{proof}The identity \[\sup_{\varepsilon > 0}(-\Im \varphi_{\nu}(i\varepsilon))/\varepsilon =\int_{\mathbb{R}}\frac{1+t^2}{t^2}\,d\sigma(t)\] follows from the free L\'{e}vy-Khintchine formula \[-\Im \varphi_{\nu}(i\varepsilon)=\varepsilon \int_{\mathbb{R}}\frac{1+t^2}{\varepsilon^{2}+t^2}\,d\sigma(t)\] and the monotone convergence theorem, and we see that the supremum here is in fact a genuine limit: \[\sup_{\varepsilon > 0}(-\Im \varphi_{\nu}(i\varepsilon))/\varepsilon=\lim_{\varepsilon \rightarrow 0^{+}}(-\Im \varphi_{\nu}(i\varepsilon))/\varepsilon.\]

Next, recall from Proposition 4.7 in \cite{BB2005} that $0 \in F_{\nu}\left(\mathbb{R}\right)$ if and only if the limit \[t_{\nu} =H_{\nu}(0)=\lim_{\varepsilon \rightarrow 0^{+}}H_{\nu}(i\varepsilon)\] exists, $t_{\nu} \in \mathbb{R}$, and the Julia-Carath\'eodory derivative $H'_{\nu}(0) \geq 0$. Note that if the limit $t_{\nu}$ exists and is real, then the derivative \begin{equation}\label{eq:5.2} H'_{\nu}(0)=\lim_{\varepsilon \rightarrow 0^{+}}\frac{\Im H_{\nu}(i\varepsilon)}{\varepsilon}\end{equation} always exists and belongs to the interval $ [-\infty,1)$. Moreover, if $0 \in F_{\nu}\left(\mathbb{R}\right)$ and $H'_{\nu}(0) >0$ then we have the Julia-Carath\'eodory derivative $F'_{\nu}\left(t_{\nu} \right)=1/H'_{\nu}(0)$.

Now, if $0 \in F_{\nu}\left(\mathbb{R}\right)$, then we know the limit $t_{\nu} \in \mathbb{R}$. Hence, (5.2) implies  $H'_{\nu}(0)=1-L$. Since $H'_{\nu}(0) \geq 0$ in this case, we conclude that $1 \geq L$. On the other hand, since $F_{\nu}\left(\mathbb{R}\right)=\partial \Omega_{\nu}$, the inversion formula shows that \[F_{\nu}\left(t_{\nu}\right)=F_{\nu}\left(H_{\nu}(0)\right)=0.\]

Conversely, if the singular integral $L$ converges and $1 \geq L$, then we have $\Im H_{\nu}(i\varepsilon) \rightarrow 0\cdot(1-L)=0$ as $\varepsilon \rightarrow 0^{+}$. On the other hand, the estimate\[\frac{|t|}{\varepsilon^2+t^2} \leq \frac{1+t^2}{\varepsilon^2+t^2}\leq \frac{1+t^2}{t^2} \in L^{1}(\sigma), \quad t \in \mathbb{R}, \quad \varepsilon >0,\]  and the dominated convergence theorem imply that the function $t \mapsto 1/t$ belongs to $L^{1}(\sigma)$ and \[\Re H_{\nu}(i\varepsilon)=\gamma+(\varepsilon^2 -1)\int_{\mathbb{R}}\frac{t}{\varepsilon^{2}+t^2}\,d\sigma(t) \rightarrow \gamma-\int_{\mathbb{R}}\frac{1}{t}\,d\sigma(t)\] as $\varepsilon \rightarrow 0^{+}$. It follows that the vertical limit \[t_{\nu}=\gamma-\int_{\mathbb{R}}\frac{1}{t}\,d\sigma(t) \in \mathbb{R}.\] As seen earlier, this fact and the formula (5.2) imply that $H'_{\nu}(0)=1-L$. Therefore, we have $H'_{\nu}(0) \geq 0$, and the proof of Part (1) is finished. 

Part (2) follows from the fact that the derivative $F'_{\nu}\left(t_{\nu} \right)=1/{\nu}\left(\{t_{\nu}\}\right)$.
\end{proof} We remark that the results in \cite{BB2005} were proved using Denjoy-Wolff analysis for boundary fixed points of analytic self-maps on $\mathbb{C}^{+}$. A different approach to the same results has been done in \cite{HW}, which yields a more general description for the points on the boundary set $\partial \Omega_{\nu}$. 

\subsection{Stable approximation} Recall that two measures $\mu,\nu \in \mathcal{M}$ are said to have the same \emph{type} (and we write $\mu \sim \nu$) if there exist constants $a>0$ and $b \in \mathbb{R}$ such that $\mu\left(E\right)=\nu\left(aE+b\right)$ for all Borel sets $E \subset \mathbb{R}$. The relation $\sim$ is an equivalence relationship among all probability laws, and hence the set $\mathcal{M}$ is partitioned into a union of distributions with inequivalent types. A nondegenerate distribution $\nu \in \mathcal{M}$ is said to be \emph{$\boxplus$-stable} if $\nu \sim \nu_{1} \boxplus \nu_{2}$ whenever $\nu_{1}\sim\nu\sim\nu_{2}$. Clearly, within one type either all distributions are stable or else none of them is stable.
 
Each $\boxplus$-stable law $\nu$ is associated with a unique  \emph{stability index} $\alpha \in (0,2]$, so that if $X$ and $Y$ are free random variables drawn from the same law $\nu$ and $a,b >0$, then the distribution of the sum $aX+bY$ is a translate of the distribution of the scaled variable $(a^{\alpha}+b^{\alpha})^{1/\alpha}X$. Apparently, all stable laws of the same type must share the same index. 

Freely stable laws are $\boxplus$-infinitely divisible and absolutely continuous, and they can be classified using the stability index $\alpha$. Following \cite{BPata1999}, every $\boxplus$-stable law has the same type as a unique distribution whose Voiculescu transform falls into the following list:
\begin{enumerate}
\item $\varphi(z)=1/z$ for $\alpha=2$;
\item $\varphi(z)=bz^{1-\alpha}$ for $1<\alpha<2$, where $|b|=1$ and $\arg b \in [(\alpha-2)\pi,0]$;
\item $\varphi(z)=bz^{1-\alpha}$ for $0<\alpha<1$, where $|b|=1$ and $\arg b \in [\pi, (1+\alpha)\pi]$;
\item $\varphi(z)=-2bi+[2(2b-1)/\pi]\log z$ for $\alpha=1$, where $b \in [0,1]$.
\end{enumerate} 
Here, the complex power and logarithmic functions are given by their principal value in $\mathbb{C}^{+}$. One can also find a formula for the density of the $\boxplus$-stable laws in \cite{BPata1999}. Among all, we mention that the case $\alpha=2$ corresponds to the stable type of the standard semicircular law.

The interest in the class of freely stable laws arises from the fact that a measure $\nu$ is $\boxplus$-stable if and only if there exist a sequence $\{X_i\}_{i=1}^{\infty}$ of identically distributed free random variables and constants $a_n>0$ and $b_n\in \mathbb{R}$ such that the distribution of the normalized sum $S_n=\sum_{i=1}^{n}(X_{i}-b_n)/a_n$ converges weakly to the law $\nu$. In this case, the common distribution of the sequence $X_i$ is said to belong to the \emph{free domain of attraction} of the stable law $\nu$. Thus, up to a change of scale and location, the distributional behavior of a large free convolution $\mu^{\boxplus n}$ for a measure $\mu$ in a free domain of attraction can be estimated using the corresponding freely stable law. 

Free domains of attraction for $\boxplus$-stable laws are determined in \cite{BPata1999}, showing that these domains of attraction coincide with their classical counterparts relative to the classical convolution. In the semicircular case, the free domain of attraction consists of all nondegenerate measures $\mu \in \mathcal{M}$ such that the truncated variance function \[H_{\mu}(x)=\int_{-x}^{x}t^2\,d\mu(t), \quad x>0,\] satisfies $\lim_{x \rightarrow \infty}H_{\mu}(cx)/H_{\mu}(x)=1$ for any given $c>0$. This is in parallel to the classical theory of central limit theorems, that is, convergence to a Gaussian law.   

With that being said, the following result shows that the quality of freely stable approximation is in fact much better than its classical counterpart. This result is stated in the general framework of triangular arrays with identical rows. 
\begin{prop} Let $\nu$ be a $\boxplus$-stable law for which the weak approximation $\mu_n^{\boxplus k_n} \rightarrow \nu$ holds. Then the measure $\mu_n^{\boxplus k_n}$ superconverges to the law $\nu$ on $\mathbb{R}$. 
\end{prop}
\begin{proof} This is a direct consequence of Theorem 4.1 and the criterion (5.1). Indeed, one has $L=\infty$ in all cases of the index $\alpha$, which implies that $0 \notin F_{\nu}\left(\mathbb{R}\right)$.   
\end{proof}

In particular, the preceding result generalizes the superconvergence for measures with finite variance in \cite{Wang2010} to the entire free domain of attraction of the semicircular law. 

Notice that stable approximation to the free sum $S_n$ could fail for any choice of constants $a_n$ and $b_n$ if the common distribution $\mu$ of the summands $X_i$ does not belong to any free domain of attraction, but even in this case one may still have weak convergence along some subsequence $S_{k_n}$. The limit $\nu$ in this situation is necessarily $\boxplus$-infinitely divisible, and hence Theorem 4.1 still applies to this case. The law $\mu$ in this case is said to belong to the \emph{free domain of partial attraction} of the law $\nu$. In fact, a probability distribution is $\boxplus$-infinitely divisible if and only if its free domain of partial attraction is nonempty. It is also well known that the domain of partial attraction of a stable law is strictly larger than its domain of attraction in both free and classical theories. We refer to the paper \cite{BPata1999} for the details of these results.

\subsection{Poisson approximation} Here we study an example of freely infinitely divisible approximation relative to Poisson type limit theorems. Let $\mu$ be an arbitrary probability measure on $\mathbb{R}$, $\mu \neq \delta_0$, and let $\lambda >0$ be a given parameter. Recall that the \emph{compound free Poisson distribution} $\nu_{\lambda, \mu}$ with rate $\lambda$ and jump distribution $\mu$ is the weak limit of \[\left[(1-\lambda/n)\delta_0+(\lambda/n)\mu\right]^{\boxplus \,n}\]  as $n\rightarrow \infty$ \cite{Basic}. The law $\nu_{\lambda, \mu}$ is $\boxplus$-infinitely divisible, and its free generating pair is given by \[\gamma=\lambda\int_{\mathbb{R}}\frac{t}{1+t^2}\,d\mu(t), \quad d\sigma(t)=\lambda\frac{t^2}{1+t^2}\,d\mu(t).\] Thus, we see immediately that the numbers $L=\lambda$ and $t_{\nu_{\lambda,\mu}}=0$ in this case, which leads further to the following result:
\begin{prop} The origin is an atom of mass $1-\lambda$ for the law $\nu_{\lambda,\mu}$ if and only if the parameter $\lambda<1$. If $\lambda>1$, then the superconvergence phenomenon in any weak approximation $\mu_n^{\boxplus k_n} \rightarrow \nu_{\lambda,\mu}$ holds globally on $\mathbb{R}$.\end{prop}

Note that the case $\mu=\delta_1$ reduces to the  approximation by Marchenko-Pastur law: \[d\nu_{\lambda,\delta_1}(t)=\begin{cases}
     \frac{\sqrt{4\lambda-(t-1-\lambda)^2}}{2\pi t}\chi(t)\,dt, &\text{if } \lambda \geq 1;\\
      (1-\lambda)\delta_0+\frac{\sqrt{4\lambda-(t-1-\lambda)^2}}{2\pi t}\chi(t)\,dt, &\text{if } 0<\lambda<1,
     \end{cases}\]
where $\chi$ stands for the indicator function of the open interval $((1-\sqrt{\lambda})^2,(1+\sqrt{\lambda})^2)$. Clearly, the law $\nu_{1,\delta_1}$ has no atom and yet  
$F_{\nu_{1,\delta_1}}(0)=0$.

\section*{Acknowledgements}
The first named author was supported in part
by a grant of the National Science Foundation. The second named author was supported by the NSERC Canada Discovery Grant RGPIN-402601. The third named author acknowledges support provided by a Leverhulme Trust Research Project Grant. He was also supported in part by NSFC (no. 11501423, no. 71301164), China Postdoctoral Science Foundation (no. 2015M570662) and
Fundamental Research Funds for the Central Universities (no. 201410500069).

\end{document}